\documentclass[12pt]{amsart}

\usepackage[left=2cm, right=2cm, top=3cm, bottom=3cm]{geometry}
\usepackage{hyperref}
\usepackage{amsmath, amsthm, amsfonts, amssymb}
\usepackage{mathtools}
\usepackage{mathrsfs}
\usepackage{longtable}
\usepackage{booktabs}
\usepackage{upgreek}
\usepackage{appendix}

\numberwithin{equation}{section}

\theoremstyle{plain}
\newtheorem{thm}{Theorem}[section]
\newtheorem*{thm*}{Theorem}
\newtheorem{coro}[thm]{Corollary}

\newtheorem{lemma}[thm]{Lemma}

\theoremstyle{definition}
\newtheorem{deff}[thm]{Definition}

\theoremstyle{remark}
\newtheorem{rema}[thm]{Remark}

\newcommand\beq{\begin{equation}}
	\newcommand\eeq{\end{equation}}

\newcommand\tbtmat[4]{\left(\begin{smallmatrix}{#1} & {#2} \\ {#3} & {#4}\end{smallmatrix}\right)}
\newcommand\tbtMat[4]{\left(\begin{matrix}{#1} & {#2} \\ {#3} & {#4}\end{matrix}\right)}

\newcommand\slZ{\mathrm{SL}_2(\mathbb{Z})}

\allowdisplaybreaks

\makeatletter
\@namedef{subjclassname@2020}{\textup{2020} Mathematics Subject Classification}
\makeatother

\begin{document}
	
	\title[Congruence for $4$-colored generalized Frobenius partitons]{Congruence families modulo powers of $7$ for $4$-colored generalized Frobenius partitions}

	\author{Kang-Yu Wang}
	\address{Department of Mathematics, Shanghai Normal University, Shanghai, People's Republic of China}
	\email{ 1000570883@smail.shnu.edu.cn }

	\author{Yi-Ning Wang}
	\address{Department of Mathematics, Shanghai Normal University, Shanghai, People's Republic of China}
	\email{ 1000570850@smail.shnu.edu.cn }
	
	\thanks{
		This work was supported by the National Key R\&D Program of China (\#2024YFA1014500) and the National Natural Science Foundation of China (\#12401438).
	}
	
	
	
	\begin{abstract}
		In 2012, Peter Paule and Cristian-Silviu Radu proved an infinite family of Ramanujan type congruences for $2$-colored Frobenius partitions $c\phi_2$ introduced by George Andrews. Recently, Frank Garvan, James Sellers and Nicolas Smoot showed that this family of congruences is equivalent to the family of congruences for $(2,0)$-colored Frobenius partitions $c\psi_{2,0}$ introduced by Brian Drake and by Yuze Jiang, Larry Rolen and Michael Woodbury for the general case. Motivated by Garvan, Sellers and Smoot's work, Rong Chen and Xiao-Jie Zhu found modular transformations relating the $c\psi_{k,\beta}$ for fixed $k$ and varying $\beta$. As an example, they proved a family of congruences for $c\psi_{3,1/2}$ following Paule and Radu's work and then proved the equivalence between $c\psi_{3,1/2}$ and $c\phi_3=c\psi_{3,3/2}$. In the present paper, we give a new example of Chen and Zhu's framework for $c\psi_{4,\beta}$. Our proof is considerably simpler.

	\end{abstract}

	\maketitle
	
	\section{Introduction}
	
	In his 1984 AMS Memoir, Andrews \cite{An84} defined the family of $k$-colored generalized Frobenius partition functions $c\phi_k(n)$. He proved a variety of Ramanujan-type congruences satisfied by these $c\phi_k(n)$ such as
	\begin{align}
		\label{r:cphi2}
		c\phi_2(5n+3)\equiv 0 \pmod{5}.
	\end{align}
	The higher-power generalization of \eqref{r:cphi2}
	\begin{align}
		\label{r:cphi2a}
		c\phi_2(5^\alpha n+\delta_\alpha)\equiv 0 \pmod{5^\alpha},
	\end{align}
	where $12\delta_\alpha\equiv 1 \pmod{5^\alpha}$ was conjectured by Sellers \cite{se94} in 1994 and proved by Paule and Radu \cite{Peter} in 2012.
	
	In 2011, Baruah and Sarmah \cite{basa11} established the following expressions for the generating function of $c\phi_4(n)$
	\begin{align}
		\label{r:cphi4g}
		\sum_{n=0}^\infty c\phi_4(n)q^n=q^{1/6}\frac{\eta(4\tau)^{15}}{\eta(8\tau)^6\eta(2\tau)^6\eta(\tau)^4}
		+12q^{1/6}\frac{\eta(8\tau)^2\eta(4\tau)^3}{\eta(2\tau)^2\eta(\tau)^4},
	\end{align}
	where $q:=\mathrm{e}^{2\pi \mathrm{i} \tau}$ and
	$$
	\eta(\tau)=q^{1/24}\prod_{j=1}^\infty (1-q^j)
	$$
	is the usual Dedekind eta function. Then the Ramanujan-type congruence of $c\phi_4(n)$ modulo $7$
	\begin{align}
		\label{r:cphi4}
		c\phi_4(7n+6)\equiv 0 \pmod{7}
	\end{align}
	was found by Zhang and Wang \cite[Theorem 1]{cpsi42} in 2017, and appeared earlier in partial form in \cite{chwaya17,lin14}. One reason why the higher-power generalization of \eqref{r:cphi4} has not been found in the literature may be that the generating function of $c\phi_4(n)$ in \eqref{r:cphi4g} appears to involve modular functions on the curve $X_0(56)$ when taken modulo power of $7$.
	
	This paper illustrates, by means of a new example, how the recent work of Chen and Zhu \cite{ncpsi} can serve as a heuristic tool for proving congruence families modulo powers of $7$ for $c\phi_4(n)$.
	
	Recently, Jiang, Rolen, and Woodbury \cite{Yu22} considered $(k,a)$-colored Frobenius partition functions $c\psi_{k,a}(n)$, which are natural generalizations of Andrews' $c\phi_k$ since $c\psi_{k,k/2}(n)=c\phi_k(n)$. Use the notation in \cite{Yu22}, Garvan, sellers, and smoot \cite{gss24} found that the property of $c\psi_{2,0}$ modulo powers of $5$ is similar to \eqref{r:cphi2a}. They obtained
	\begin{align}
		\label{r:cpsi2}
		c\psi_{2,\beta}(n)\equiv 0 \pmod{5^\alpha},
	\end{align}
	where $\beta=0,1$ and $3\beta^2-2\equiv 12 n \pmod{5^\alpha}$. (We used an equivalent condition on $n$ for uniformity, and for $\beta=1$ this is just \eqref{r:cphi2a}.) They also realized that the two families of congruences \eqref{r:cpsi2} bear a strange symmetry between each other via the application of certain Atkin-Lehner involution to their generating functions. Motivated by \cite{gss24}, Chen and Zhu \cite{ncpsi} found modular transformations relating $c\psi_{k,\beta}$ for fixed $k$ and varying $\beta$, which provides a more useful approach to proving Ramanujan-type congruences. As an example, they proved that
	\begin{align}
		\label{r:cpsi3}
		c\psi_{3,\beta}(n)\equiv 0 \pmod{5^\alpha},
	\end{align}
	where $\beta=1/2,3/2$ and $4\beta^2-6\equiv 24 n \pmod{5^{2\alpha}}$. (See \cite[Theorem 1.4]{ncpsi}.)
	
	The main result in this paper is the following.
	
	\begin{thm} \label{thm:cpsi4012}
		For $\alpha \in \mathbb{N}^*$, $\beta \in \{0,1,2\}$ and $n \in \mathbb{N}^*$ such that $3\beta^2-8 \equiv 24n \pmod{7^{2\alpha-1}}$, we
		have
		\begin{align}
			\label{r:cpsi4}
			c\psi_{4,\beta}(n) \equiv 0 \pmod{7^{\alpha}}.
		\end{align}
	\end{thm}
	\begin{rema}
		Letting $\beta=2$ in Theorem \ref{thm:cpsi4012}, we have the higher-power generalization of \eqref{r:cphi4}
		\begin{align*}
			c\phi_4(7^\alpha n+\delta_\alpha)\equiv 0 \pmod{7^{\lfloor\frac{\alpha+1}{2}\rfloor}},
		\end{align*}
		where $6\delta_\alpha\equiv 1 \pmod{7^\alpha}$.
	\end{rema}
	Define the generating function $C\Phi_{k,\beta}(q)$ as follows:
	$$
	C \Psi_{k,\beta}(q):=\sum_{n=0}^\infty c\psi_{k,\beta}(n)q^n.
	$$
	According to \cite[Theorem 2]{Yu22}, one can obtain that
	\begin{align}
		\label{r:c21n}
		C \Psi_{2,1}(q)=q^{1/12}\frac{\eta_2^5}{\eta_4^2 \eta_1^4},
	\end{align}
	\begin{align}
		\label{r:c31n}
		C \Psi_{3,1/2}(q)=3q^{-5/24}\frac{\eta_3^3}{\eta_1^4},
	\end{align}
	and
	\begin{align}
		\label{r:c41n}
		C \Psi_{4,1}(q)=4q^{-5/24}\frac{\eta_2^6}{\eta_1^7},
	\end{align}
	where $\eta_{k}$ denotes the shorthand notation  derived from Dedekind's eta function
	$$
	\eta_{k}:=\eta(k\tau).
	$$
	
	When the generating function can be modified to a modular function on a simple modular curve (usually of genus no more than 1 and with no more than 6 cusps), a typically straightforward method for proving Ramanujan-type congruences modulo powers of a prime, illustrated by the generating functions \eqref{r:c21n} -- \eqref{r:c41n}, is that given by Paule and Radu \cite{Peter}, where the approach for handling $c\psi_{2,1}$ was presented. The function $c\psi_{3,1/2}$ was subsequently treated by the same standard method and appeared in \cite{ncpsi}. In this paper we also employ this technique to prove a family of congruences for $c\psi_{4,1}$ modulo powers of $7$. The generating functions of $c\psi_{3,3/2}$, $c\psi_{4,0}$, and $c\psi_{4,2}$ appear to be intricate. Nevertheless, using \cite[Example 1.2]{ncpsi}
	, one can derive that (noting the symmetry between \eqref{r:c21n} -- \eqref{r:c41n} and \eqref{r:c2n} -- \eqref{r:c4n})
	\begin{align}
		\label{r:c2n}
		C \Psi_{2,1}(q^4)+q C \Psi_{2,0}(q^4)=q^{1/3}\frac{\eta_2^5}{\eta_1^2 \eta_4^4},
	\end{align}
	\begin{align}
		\label{r:c3n}
		C \Psi_{3,3/2}(q^3)-q C \Psi_{3,1/2}(q^3)=q^{3/8}\frac{\eta_1^3}{\eta_3^4},
	\end{align}
	and
	\begin{align}
		\label{r:c4n}
		C \Psi_{4,2}(q^2)-q C \Psi_{4,0}(q^2)=q^{1/3}\frac{\eta_1^6}{\eta_2^7}.
	\end{align}
	The detailed proof of \eqref{r:c4n} is given in Section \ref{section:3} (see Lemma \ref{lemma012}). We note that \eqref{r:c2n} -- \eqref{r:c4n} can also be proved by \cite[Theorem 2]{Yu22} and modular form techniques such as the valence formula. The generating functions \eqref{r:c2n} -- \eqref{r:c4n} show that all families of congruences \eqref{r:cpsi2} -- \eqref{r:cpsi4} already fall within the standard approach.
	
	However, one can easily obtain the generating functions \eqref{r:c2n} -- \eqref{r:c4n} via the modular transformation introduced by Chen and Zhu \cite{ncpsi} and thereby link the proofs for each family of congruences in \eqref{r:cpsi2}, \eqref{r:cpsi3}, and \eqref{r:cpsi4} as presented in \cite{gss24}. We note that the generating functions \eqref{r:c21n} -- \eqref{r:c41n} and \eqref{r:c2n} -- \eqref{r:c4n} are associated via certain clear Atkin-Lehner involutions, respectively. The involution between \eqref{r:c21n} and \eqref{r:c2n} is an alternative to that presented in \cite{gss24} and also helpful for relating the two families of congruences \eqref{r:cpsi2} in a similar way. The involution between \eqref{r:c31n} and \eqref{r:c3n} yields a simple proof of \eqref{r:cpsi3} for $\beta=3/2$ whereas the proof in \cite{ncpsi} is more general. In this paper we prove \eqref{r:cpsi4} for $\beta=0,2$ by using the involution between \eqref{r:c41n} and \eqref{r:c4n}. \\
	
	This paper is organized as follows. In Section \ref{section:2}, we prove Theorem \ref{thm:cpsi4012} for $\beta=1$ following \cite{Peter}.
	In Section \ref{section:3}, we obtain some results for the Atkin-Lehner involution among the families of congruences in Theorem \ref{thm:cpsi4012} and then prove Theorem \ref{thm:cpsi4012} for $\beta=0,2$.

	\section{the congruence family of $c\psi_{4,1}$} \label{section:2}
	
	The proof relies on the 42 identities that appeared in Appendix, which can be verified by the valence formula from modular form theory. Garvan has written a MAPLE package called ETA which can prove these identities automatically. See
	\begin{align*}
		https ://qseries.org/fgarvan/qmaple/ETA/
	\end{align*}
	The tutorial for this see \cite{ga19}.
	Let
	\begin{align} \label{A=eta}
		A:=q^{-10} \frac{C \Psi_{4,1}(q)}{C \Psi_{4,1}(q^{49})}=
		\frac{\eta_2^6 \eta_{49}^7}{\eta_1^7 \eta_{98}^6},
	\end{align}
	and construct the function sequence $\{ L_{\alpha} \}_{\alpha=0}^\infty $:
	\begin{align*}
		L_0:=1 , \quad
		L_{2\alpha+1}:= U_7
		\big( A  L_{2\alpha}  \big)
		, \quad
		L_{2\alpha+2}:= U_7
		\big(  L_{2\alpha+1} \big),  \quad \quad \alpha \in \mathbb{N}.
	\end{align*}where we define $U_m (f) : \mathbb{H} \to \mathbb{C}$ as follows:
	\begin{align*}
		U_m (f)(\tau) &:= \frac{1}{m} \sum_{\lambda=0}^{m-1}
		f\Big( \frac{\tau+\lambda}{m} \Big), \quad  m \in \mathbb{N}^* , \ f : \mathbb{H} \to \mathbb{C} .
	\end{align*}

	By mathematical induction, for $\alpha \in \mathbb{N}^*$ we can easily prove that
	\begin{align}
		L_{2\alpha-1}
		&=\frac{1}{4q}
		\prod_{j=1}^\infty
		\frac{(1-q^{7j})^7}{(1-q^{14j})^6}
		\sum_{n=0}^\infty
		c \psi_{4,1} \bigg(7^{2\alpha-1}n + \lambda_{2\alpha-1} \bigg)
		q^n,  \label{L2alpha-1} \\
		L_{2\alpha}
		&=\frac{1}{4}
		\prod_{j=1}^\infty
		\frac{(1-q^{j})^7}{(1-q^{2j})^6}
		\sum_{n=0}^\infty
		c \psi_{4,1} \bigg(7^{2\alpha}n + \lambda_{2\alpha} \bigg)
		q^n, \label{L2alpha}
	\end{align}
	where 
	\begin{align}
		\label{r:lam1}
		\lambda_{2\alpha-1}
		:= \frac{-5+11 \cdot 7^{2\alpha-1}}{24},
		\ 	\  \ \
		\lambda_{2\alpha}
		:= \frac{-5+5 \cdot 7^{2\alpha}}{24}, \ \ \forall \  \alpha \in \mathbb{N}^*.
	\end{align}

	To obtain another expression for $\{ L_{\alpha}  \}_{\alpha=1}^\infty $, let
	\begin{align}
		t:=     \frac{\eta_7^4}{\eta_1^4} ,
		\ \ \ \
		p_1:=   \frac{\eta_{14}^4\eta_1^4}{\eta_7^4\eta_2^4} ,
		\ \ \ \
		p_2:=   \frac{\eta_7^3\eta_1^3}{\eta_{14}^3\eta_2^3} -p_1.
		\label{t,p=eta}
	\end{align}

	\begin{lemma}[{\cite[Lemma 3.2]{chen7}}] \label{lemma2.0}
		Suppose $u=u(\tau)$, and $k$ is any integer. Then
		\begin{align*}
			U_7(ut^k)= \sum_{j=0}^{6} a_j(t)U_7(ut^{k+j-7}),
		\end{align*}
		where
		\begin{align*}
			a_j(t)=\sum_{l=1}^7 s(j,l) 7^{ \big\lfloor \frac{7l+j-4}{4} \big\rfloor } t^l  \ \ \ , j=0,1,\cdots,6,
		\end{align*}
		with certain integers $s(j,l)$.
	\end{lemma}
	
	\begin{lemma}[{\cite[Lemma 3.3]{chen7}}] \label{lemma2.1}
		Let $v_1,v_2,v_3,u:\mathbb{H} \to \mathbb{C}$ and $l \in \mathbb{Z}$. Suppose for $l \le k \le l+6 $ there exist Laurent polynomials  $p_k^{(1)}(t),p_k^{(2)}(t),p_k^{(3)}(t) \in \mathbb{Z}[t,t^{-1}] $ such that \\
		\begin{align}
			U_7(ut^k)= v_1 p_k^{(1)}(t)  +v_2 p_k^{(2)}(t) +v_3 p_k^{(3)}(t) \label{eq1}
		\end{align}
		and
		\begin{align}
			\mathrm{ord}_{t} \bigg(p_k^{(i)}(t) \bigg) \ge \bigg\lceil \frac{k+s_i}{7}   \bigg\rceil , \ \ \ i \in \{ 1,2,3 \} \label{eq2}
		\end{align}
		for some fixed integers $s_1,s_2$ and $s_3$. Then exist families of Laurent polynomials  $p_k^{(1)}(t),p_k^{(2)}(t),p_k^{(3)}(t) \in \mathbb{Z}[t,t^{-1}] ,k \in \mathbb{Z}$, such that \eqref{eq1} and \eqref{eq2} hold for all $k \in \mathbb{Z}$.
	\end{lemma}

	\begin{lemma} [{\cite[Lemma 3.4]{chen7}}] \label{lemma2.2} Let $v_1,v_2,v_3,u:\mathbb{H} \to \mathbb{C}$ and $l \in \mathbb{Z}$. Suppose for $l \le k \le l+6 $ there exist Laurent polynomials  $p_k^{(1)}(t),p_k^{(2)}(t),p_k^{(3)}(t) \in \mathbb{Z}[t,t^{-1}] $  such that
		\begin{align}
			U_7(ut^k)= v_1 p_k^{(1)}(t)  +v_2 p_k^{(2)}(t) +v_3 p_k^{(3)}(t) \label{eq3}
		\end{align}
		where
		\begin{align}
			p_k^{(i)}(t)= \sum_n c_{k,n}^{(i)} 7^{ \big\lfloor \frac{7n-k+\gamma_i}{4} \big\rfloor } t^n \label{eq4}
		\end{align}
		with integers $\gamma_i$ and $c_{k,n}^{(i)}$.
		Then there exist families of Laurent polynomials $p_{k}^{(i)}(t) \in \mathbb{Z}[t,t^{-1}],k \in \mathbb{Z}$, of the form $\eqref{eq4}$  for which property $\eqref{eq3}$ holds for all $k \in \mathbb{Z}$.
	\end{lemma}

	\begin{lemma} \label{lemma2.3} Given $A$ as in \eqref{A=eta}, and $t$, $p_1$, $p_2$ as in $\eqref{t,p=eta}$. Then exist discrete arrays such that the following relations hold for all $k \in \mathbb{N}$:
		\begin{align*}
			&U_7(At^k) \ \ \
			=  \sum_{n \ge \big\lceil \frac{k+4}{7} \big\rceil }
			c_{k,n}^{(a,1)} 7^{ \big\lfloor \frac{7n-k-3}{4} \big\rfloor } t^n
			+p_1 \sum_{n \ge \big\lceil \frac{k-3}{7}\big\rceil }
			c_{k,n}^{(a,2)} 7^{ \big\lfloor \frac{7n-k-2}{4} \big\rfloor } t^n
			+p_2 \sum_{n \ge \big\lceil \frac{k-3}{7}\big\rceil }
			c_{k,n}^{(a,3)} 7^{ \big\lfloor \frac{7n-k-6}{4} \big\rfloor } t^n,  \\
			&U_7(Ap_1t^k)
			= \sum_{n \ge \big\lceil \frac{k+5}{7}\big\rceil }
			c_{k,n}^{(b,1)} 7^{ \big\lfloor \frac{7n-k-2}{4} \big\rfloor } t^n
			+p_1 \sum_{n \ge \big\lceil \frac{k-2}{7}\big\rceil }
			c_{k,n}^{(b,2)} 7^{ \big\lfloor \frac{7n-k-3}{4} \big\rfloor } t^n
			+p_2 \sum_{n \ge \big\lceil \frac{k-2}{7}\big\rceil }
			c_{k,n}^{(b,3)} 7^{ \big\lfloor \frac{7n-k-6}{4} \big\rfloor } t^n,    \\
			&U_7(Ap_2t^k)
			= \sum_{n \ge \big\lceil \frac{k+3}{7}\big\rceil }
			c_{k,n}^{(c,1)} 7^{ \big\lfloor \frac{7n-k-2}{4} \big\rfloor } t^n
			+p_1 \sum_{n \ge \big\lceil \frac{k-4}{7}\big\rceil }
			c_{k,n}^{(c,2)} 7^{ \big\lfloor \frac{7n-k-2}{4} \big\rfloor } t^n
			+p_2 \sum_{n \ge \big\lceil \frac{k-4}{7}\big\rceil }
			c_{k,n}^{(c,3)} 7^{ \big\lfloor \frac{7n-k-6}{4} \big\rfloor } t^n,    \\
			&U_7(t^k)  \ \ \ \ \
			= \sum_{n \ge \big\lceil \frac{k}{7}\big\rceil }
			c_{k,n}^{(d,1)} 7^{ \big\lfloor \frac{7n-k-1}{4} \big\rfloor } t^n,  \\
			&U_7(p_1t^k) \ \
			= \sum_{n \ge \big\lceil \frac{k}{7}\big\rceil }
			c_{k,n}^{(e,1)} 7^{ \big\lfloor \frac{7n-k}{4} \big\rfloor } t^n
			+p_1 \sum_{n \ge \big\lceil \frac{k+3}{7}\big\rceil }
			c_{k,n}^{(e,2)} 7^{ \big\lfloor \frac{7n-k-1}{4} \big\rfloor } t^n
			+p_2 \sum_{n \ge \big\lceil \frac{k+7}{7}\big\rceil }
			c_{k,n}^{(e,3)} 7^{ \big\lfloor \frac{7n-k-4}{4} \big\rfloor } t^n,    \\
			&U_7(p_2t^k) \ \
			=  \sum_{n \ge \big\lceil \frac{k}{7}\big\rceil }
			c_{k,n}^{(f,1)} 7^{ \big\lfloor \frac{7n-k}{4} \big\rfloor } t^n
			+p_1 \sum_{n \ge \big\lceil \frac{k+3}{7} \big\rceil }
			c_{k,n}^{(f,2)} 7^{ \big\lfloor \frac{7n-k}{4} \big\rfloor } t^n
			+p_2 \sum_{n \ge \big\lceil \frac{k+6}{7}\big\rceil  }
			c_{k,n}^{(f,3)} 7^{ \big\lfloor \frac{7n-k-4}{4} \big\rfloor } t^n.
		\end{align*}
	\end{lemma}
	\begin{proof}
		Our Appendix shows the cases of $k \in \{0,-1,\cdots,-6\}$. Applying Lemma \ref{lemma2.1} and Lemma \ref{lemma2.2} immediately proves the statement for all $k \ge 0$.
	\end{proof}
	
	\begin{rema}
		In addition, all the coefficients of $t^n,p_1t^n,p_2t^n$ in Lemma \ref{lemma2.3} are integers. The Appendix shows the cases of $k \in \{0,-1,\cdots,-6\}$. The statement for $k \ge 0$ is given by Lemma \ref{lemma2.0}.
	\end{rema}
	
	\begin{deff} We define {\small
			\begin{align*}
				X^{(0)} &:= \bigg\{
				\ \  \ \ \  r_1(0) +7r_1(1) \ t \ \ +
				\sum_{n=2}^\infty r_1(n)  7^{ \lfloor \frac{7n-2}{4} \rfloor } t^n \\
				&  \quad \quad \quad \quad \quad \quad
				+7r_2(1)p_1t+
				\sum_{n=2}^\infty r_2(n)  7^{ \lfloor \frac{7n-2}{4} \rfloor } p_1 t^n \\
				&  \quad \quad \quad \quad \quad \quad
				+r_3(1)p_2t+
				\sum_{n=2}^\infty r_3(n)  7^{ \lfloor \frac{7n-6}{4} \rfloor } p_2 t^n
				\bigg\},
			\end{align*}
			and
			\begin{align*}
				X^{(1)} &:= \bigg\{
				\quad \quad \quad \quad \quad
				+\ r_1(1) \ t \  +
				\sum_{n=2}^\infty r_1(n)  7^{ \lfloor \frac{7n-7}{4} \rfloor } t^n \\
				&  \ \ \ \ \ \ \ \ \ +r_2(0)p_1+r_2(1)p_1t+
				\sum_{n=2}^\infty r_2(n)  7^{ \lfloor \frac{7n-7}{4} \rfloor } p_1 t^n \\
				&  \ \ \ \ \ \ \ \ \ +r_3(0)p_2+r_3(1)p_2t+
				\sum_{n=2}^\infty r_3(n)  7^{ \lfloor \frac{7n-11}{4} \rfloor } p_2 t^n
				\bigg\},
		\end{align*} }
		where $r_1,r_2,r_3:\mathbb{Z} \to \mathbb{Z}$ are arbitrarily selected discrete functions.
	\end{deff}
	\begin{thm}
		We have
		\begin{align}
			&f \in X^{(0)} \ \ \ \text{ implies } \ \ \ 7^{-1} U_7(Af) \in X^{(1)}, \label{lemma0to1}
		\end{align}
		and
		\begin{align}
			&f \in X^{(1)} \ \ \ \text{ implies } \ \ \   U_7(f) \in X^{(0)}. \label{lemma1to0}
		\end{align}
	\end{thm}
	\begin{proof}
		Assume that $f \in X^{(1)}$. Then there are discrete functions $r_1(n),r_2(n),r_3(n)$ such that
		\begin{align*}
			f&= \sum_{n=2}^\infty r_1(n)  7^{ \big\lfloor \frac{7n-7}{4} \big\rfloor } t^n
			+ \sum_{n=2}^\infty r_2(n)  7^{ \big\lfloor \frac{7n-7}{4} \big\rfloor } p_1t^n
			+ \sum_{n=2}^\infty r_3(n)  7^{ \big\lfloor \frac{7n-11}{4} \big\rfloor } p_2t^n \\
			&\ \ \ \ \ \quad \quad + r_1(1)t \ \ \ \ \ +r_2(0)p_1+r_2(1)p_1t
			\ \ \ \ \ +r_3(0)p_2+r_3(1)p_2t.
		\end{align*}
		This implies that
		\begin{align*}
			U_7(f)&= \sum_{n=2}^\infty r_1(n)
			7^{ \big\lfloor \frac{7n-7}{4} \big\rfloor }
			U_7(t^n)
			+ \sum_{n=2}^\infty r_2(n)
			7^{ \big\lfloor \frac{7n-7}{4} \big\rfloor }
			U_7(p_1t^n)
			+\sum_{n=2}^\infty r_3(n)
			7^{ \big\lfloor \frac{7n-11}{4} \big\rfloor }
			U_7(p_2 t^n) \\
			&  \ \ \ \quad \quad + r_1(1)U_7(t) \ \ \  +r_2(0)U_7(p_1)+r_2(1)U_7(p_1t)
			\ \ \ +r_3(0)U_7(p_2)+r_3(1)U_7(p_2t).
		\end{align*}
		Utilizing Lemma \ref{lemma2.3}  gives
		\begin{align*}
			U_7(f)=
			&\ \ \
			\sum_{m \ge 1 }\sum_{n \ge 2}
			r_1(n)  c_{n,m}^{(d,1)}
			7^{ \big\lfloor \frac{7n-7}{4} \big\rfloor }
			7^{ \big\lfloor \frac{7m-n-1}{4} \big\rfloor } t^m  \\
			&\quad  \quad  \quad  \quad  \quad  \quad  \quad  \quad
			+  \sum_{m \ge 1 }\sum_{n \ge 2}
			r_2(n) c_{n,m}^{(e,1)}
			7^{ \big\lfloor \frac{7n-7}{4} \big\rfloor }
			7^{ \big\lfloor \frac{7m-n}{4} \big\rfloor } t^m
			\\
			&\quad  \quad  \quad  \quad  \quad  \quad  \quad  \quad
			+ \sum_{m \ge 1 }\sum_{n \ge 2}
			r_2(n) c_{n,m}^{(e,2)}
			7^{ \big\lfloor \frac{7n-7}{4} \big\rfloor }
			7^{ \big\lfloor \frac{7m-n-1}{4} \big\rfloor } p_1t^m
			\\
			&\quad  \quad  \quad  \quad  \quad  \quad  \quad  \quad
			+ \sum_{m \ge 2 }\sum_{n \ge 2}
			r_2(n) c_{n,m}^{(e,3)}
			7^{ \big\lfloor \frac{7n-7}{4} \big\rfloor }
			7^{ \big\lfloor \frac{7m-n-4}{4} \big\rfloor } p_2t^m   \\
			&\quad  \quad  \quad  \quad  \quad  \quad  \quad  \quad  \quad  \quad  \quad  \quad  \quad \quad \quad  \quad
			+ \sum_{m \ge 1 }\sum_{n \ge 2}
			r_3(n)  c_{n,m}^{(f,1)}
			7^{ \big\lfloor \frac{7n-11}{4} \big\rfloor }
			7^{ \big\lfloor \frac{7m-n}{4} \big\rfloor } t^m
			\\
			&\quad  \quad  \quad  \quad  \quad  \quad  \quad  \quad  \quad  \quad  \quad  \quad  \quad \quad \quad  \quad
			+ \sum_{m \ge 1 }\sum_{n \ge 2}
			r_3(n)  c_{n,m}^{(f,2)}
			7^{ \big\lfloor \frac{7n-11}{4} \big\rfloor }
			7^{ \big\lfloor \frac{7m-n}{4} \big\rfloor } p_1 t^m
			\\
			& \quad  \quad  \quad  \quad  \quad  \quad  \quad  \quad  \quad  \quad  \quad  \quad  \quad \quad \quad  \quad
			+ \sum_{m \ge 2  }\sum_{n \ge 2}
			r_3(n)  c_{n,m}^{(f,3)}
			7^{ \big\lfloor \frac{7n-11}{4} \big\rfloor }
			7^{ \big\lfloor \frac{7m-n-4}{4} \big\rfloor } p_2t^m  \\
			&\quad  \quad  + r_1(1)U_7(t) \ \ \  +r_2(0)U_7(p_1)+r_2(1)U_7(p_1t)
			\ \ \  +r_3(0)U_7(p_2)+r_3(1)U_7(p_2t).
		\end{align*}
		If $n \ge 2$, $m \ge 2$, we can use $\big\lfloor \frac{I_1}{4} \big\rfloor
		+ \big\lfloor \frac{I_2}{4} \big\rfloor
		\ge  \big\lfloor \frac{I_1+I_2-3}{4} \big\rfloor $ to know that the coefficients of  $t^m,p_1t^m,p_2t^m$ satisfy the definition of $X^{(0)}$. ( Check the following seven inequalities. )
		{\small
			\begin{align*}
				&  \bigg\lfloor \frac{7n-7}{4} \bigg\rfloor
				+ \bigg\lfloor \frac{7m-n-1}{4} \bigg\rfloor
				\ge  \bigg\lfloor \frac{7m-2}{4} \bigg\rfloor;  \\
				&\quad  \quad  \quad  \quad  \quad  \quad  \quad  \quad   \bigg\lfloor \frac{7n-7}{4} \bigg\rfloor
				+\ \  \bigg\lfloor \frac{7m-n}{4} \bigg\rfloor \ \ \
				\ge  \bigg\lfloor \frac{7m-2}{4} \bigg\rfloor,  \\
				&\quad  \quad  \quad  \quad  \quad  \quad  \quad  \quad  \bigg\lfloor \frac{7n-7}{4} \bigg\rfloor
				+ \bigg\lfloor \frac{7m-n-1}{4} \bigg\rfloor
				\ge  \bigg\lfloor \frac{7m-2}{4} \bigg\rfloor,  \\
				&\quad  \quad  \quad  \quad  \quad  \quad  \quad  \quad   \bigg\lfloor \frac{7n-7}{4} \bigg\rfloor
				+ \bigg\lfloor \frac{7m-n-4}{4} \bigg\rfloor
				\ge \bigg\lfloor \frac{7m-6}{4} \bigg\rfloor;  \\
				&\quad  \quad  \quad  \quad  \quad  \quad  \quad  \quad  \quad  \quad  \quad  \quad  \quad \quad \quad  \quad   \bigg\lfloor \frac{7n-11}{4} \bigg\rfloor
				+\ \  \bigg\lfloor \frac{7m-n}{4} \bigg\rfloor \ \ \
				\ge \bigg\lfloor \frac{7m-2}{4} \bigg\rfloor,  \\
				&\quad  \quad  \quad  \quad  \quad  \quad  \quad  \quad  \quad  \quad  \quad  \quad  \quad \quad \quad  \quad   \bigg\lfloor \frac{7n-11}{4} \bigg\rfloor
				+\ \  \bigg\lfloor \frac{7m-n}{4} \bigg\rfloor  \ \ \
				\ge \bigg\lfloor \frac{7m-2}{4} \bigg\rfloor,  \\
				&\quad  \quad  \quad  \quad  \quad  \quad  \quad  \quad  \quad  \quad  \quad  \quad  \quad \quad \quad  \quad   \bigg\lfloor \frac{7n-11}{4} \bigg\rfloor
				+ \bigg\lfloor \frac{7m-n-4}{4} \bigg\rfloor
				\ge \bigg\lfloor \frac{7m-6}{4} \bigg\rfloor .
			\end{align*}
		}
		If $n \ge 2$, $m =1 $, we  have  checked the coefficients of $t,p_1t,p_2t$. ( Check the following five inequalities. )
		{\small
			\begin{align*}
				&  \bigg\lfloor \frac{7n-7}{4} \bigg\rfloor
				+ \bigg\lfloor \frac{7m-n-1}{4} \bigg\rfloor
				\ge   1;  \\
				&\quad  \quad  \quad  \quad  \quad  \quad  \quad  \quad   \bigg\lfloor \frac{7n-7}{4} \bigg\rfloor
				+\ \  \bigg\lfloor \frac{7m-n}{4} \bigg\rfloor \ \ \
				\ge   1, \\
				&\quad  \quad  \quad  \quad  \quad  \quad  \quad  \quad  \bigg\lfloor \frac{7n-7}{4} \bigg\rfloor
				+ \bigg\lfloor \frac{7m-n-1}{4} \bigg\rfloor
				\ge   1;\\
				&\quad  \quad  \quad  \quad  \quad  \quad  \quad  \quad  \quad  \quad  \quad  \quad  \quad \quad \quad  \quad   \bigg\lfloor \frac{7n-11}{4} \bigg\rfloor
				+\ \  \bigg\lfloor \frac{7m-n}{4} \bigg\rfloor \ \ \
				\ge  1, \\
				&\quad  \quad  \quad  \quad  \quad  \quad  \quad  \quad  \quad  \quad  \quad  \quad  \quad \quad \quad  \quad   \bigg\lfloor \frac{7n-11}{4} \bigg\rfloor
				+\ \  \bigg\lfloor \frac{7m-n}{4} \bigg\rfloor  \ \ \
				\ge  1.
			\end{align*}
		}
		If $n = 1$, we  have checked the coefficients of  $t^m,p_1t^m,p_2t^m$ which come from $U_7(t), U_7(p_1t)$, and $U_7(p_2t)$, under the conditions of $m \ge 2$, $m=1$.  ( Check the following inequalities. )
		{\small
			\begin{align*}
				m\ge 2: \quad \quad &  0
				+ \bigg\lfloor \frac{7m-n-1}{4} \bigg\rfloor
				\ge  \bigg\lfloor \frac{7m-2}{4} \bigg\rfloor;  \\
				&\quad  \quad  \quad  \quad  \quad  \quad  \quad  \quad
				0
				+\ \  \bigg\lfloor \frac{7m-n}{4} \bigg\rfloor \ \ \
				\ge  \bigg\lfloor \frac{7m-2}{4} \bigg\rfloor,  \\
				&\quad  \quad  \quad  \quad  \quad  \quad  \quad  \quad
				0
				+ \bigg\lfloor \frac{7m-n-1}{4} \bigg\rfloor
				\ge  \bigg\lfloor \frac{7m-2}{4} \bigg\rfloor,  \\
				&\quad  \quad  \quad  \quad  \quad  \quad  \quad  \quad
				0
				+ \bigg\lfloor \frac{7m-n-4}{4} \bigg\rfloor
				\ge \bigg\lfloor \frac{7m-5}{4} \bigg\rfloor;  \\
				&\quad  \quad  \quad  \quad  \quad  \quad  \quad  \quad  \quad  \quad  \quad  \quad  \quad \quad \quad  \quad
				0
				+\ \  \bigg\lfloor \frac{7m-n}{4} \bigg\rfloor \ \ \
				\ge \bigg\lfloor \frac{7m-2}{4} \bigg\rfloor,  \\
				&\quad  \quad  \quad  \quad  \quad  \quad  \quad  \quad  \quad  \quad  \quad  \quad  \quad \quad \quad  \quad
				0
				+\ \ \bigg\lfloor \frac{7m-n}{4} \bigg\rfloor \ \ \
				\ge \bigg\lfloor \frac{7m-2}{4} \bigg\rfloor,  \\
				&\quad  \quad  \quad  \quad  \quad  \quad  \quad  \quad  \quad  \quad  \quad  \quad  \quad \quad \quad  \quad
				0
				+ \bigg\lfloor \frac{7m-n-4}{4} \bigg\rfloor
				\ge \bigg\lfloor \frac{7m-5}{4} \bigg\rfloor;  \\
				m = 1: \quad \quad &  0
				+ \bigg\lfloor \frac{7m-n-1}{4} \bigg\rfloor
				\ge  1; \\
				&\quad  \quad  \quad  \quad  \quad  \quad  \quad  \quad
				0
				+\ \  \bigg\lfloor \frac{7m-n}{4} \bigg\rfloor \ \ \
				\ge  1,  \\
				&\quad  \quad  \quad  \quad  \quad  \quad  \quad  \quad
				0
				+ \bigg\lfloor \frac{7m-n-1}{4} \bigg\rfloor
				\ge 1; \\
				&\quad  \quad  \quad  \quad  \quad  \quad  \quad  \quad  \quad  \quad  \quad  \quad  \quad \quad \quad  \quad
				0
				+\ \  \bigg\lfloor \frac{7m-n}{4} \bigg\rfloor \ \ \
				\ge 1,   \\
				&\quad  \quad  \quad  \quad  \quad  \quad  \quad  \quad  \quad  \quad  \quad  \quad  \quad \quad \quad  \quad
				0
				+\ \ \bigg\lfloor \frac{7m-n}{4} \bigg\rfloor \ \ \
				\ge 1.
			\end{align*}
		}
		If $n = 0$, we have checked the coefficients of  $t^m,p_1t^m,p_2t^m$ which come from $U_7(p_1), U_7(p_2)$ given in the Appendix. Hence $U(f) \in X^{(0)}$.
		
		Now assume that $f \in X^{(0)}$. Then there are discrete functions $r_1(n),r_2(n),r_3(n)$ such that
		\begin{align*}
			f&= \sum_{n=2}^\infty r_1(n)  7^{ \big\lfloor \frac{7n-2}{4} \big\rfloor } t^n
			+p_1 \sum_{n=2}^\infty r_2(n)  7^{ \big\lfloor \frac{7n-2}{4} \big\rfloor } t^n
			+p_2 \sum_{n=2}^\infty r_3(n)  7^{ \big\lfloor \frac{7n-6}{4} \big\rfloor } t^n \\
			&\ \ \ \ \ + r_1(0) + 7r_1(1)t \ \ \ \ \ \quad \quad  +7r_2(1)p_1t
			\ \ \ \ \ \quad \quad +r_3(1)p_2t.
		\end{align*}

		Similar to the proof of \eqref{lemma1to0}, for $n \ge 2$, $n =1$, $n=0$, we need to check the coefficients of  $t^m,p_1t^m,p_2t^m$ which come from $U_7(At^n)$, $ U_7(Ap_1t^n)$, $ U_7(Ap_2t^n)$, under the conditions of $m \ge 2$, $m=1$, and $m=0$.
		However, some inequations cannot be proved in this way directly, we need to check the coefficients of $p_1,p_2$, and $p_2t$ which come from $U_7(Ap_2t)$ additionally.
		We calculate $U_7(Ap_2t)$ by Lemma \ref{lemma2.0}:
		\begin{align*}
			U_7(Ap_2t)
			&= \cdots \cdots   \\
			&+ p_2 t  \cdot \Big(8 \cdot 7 \cdot 82 \cdot 7^2 -2 \cdot 176 \cdot7^2 +
			36 \cdot 845 \cdot 7 -384 \cdot272 \cdot 7 +46\cdot7^3-12\cdot4\cdot7^3 \\
			& \quad \quad  \quad \quad      +2682\cdot46\cdot7 +100\cdot7^2     -1924\cdot7\cdot4\cdot7+4278\cdot7 \Big)
			\\
			&+p_2 \cdot \Big(272 \cdot 7 - 12 \cdot 46 \cdot 7 +100\cdot 4\cdot 7-97\cdot7 \Big)
			\\
			&+p_1 \cdot \Big(272 \cdot 7 - 12 \cdot 46 \cdot 7 +100\cdot 4\cdot 7-97\cdot7 \Big).
		\end{align*}
		These parts of $U_7(Ap_2t)$ can be divided by $7$.
		
		In conclusion, $7^{-1} U_7(Af) \in X^{(1)}$.
	\end{proof}
	
	\begin{coro}
		\label{r:col}
		For each $\alpha \ge 1$ there exist $f_{2 \alpha-1 }\in X^{(1)}$ and $f_{2 \alpha }\in X^{(0)}$ such that
		\begin{align}
			L_{2 \alpha-1}= 7^{\alpha} f_{2\alpha-1}, \quad \quad
			L_{2 \alpha}= 7^{\alpha} f_{2\alpha}.
		\end{align}
	\end{coro}
	

Then we arrive at Theorem \ref{thm:cpsi41}.

\begin{thm} \label{thm:cpsi41}
	For $\alpha \in \mathbb{N}^*$ and $n \in \mathbb{N}$, we have
	\begin{align*}
		c\psi_{4,1}(7^\alpha n+\lambda_\alpha) \equiv 0 \pmod{7^{\lfloor\frac{\alpha+1}{2}\rfloor}},
	\end{align*}
	where $\lambda_\alpha$ defined by \eqref{r:lam1}.
\end{thm}

\section{from $c\psi_{4,1}$ to $c\psi_{4,0}$ and $c\psi_{4,2}$} \label{section:3}

Let
\begin{align*}	
	A':=
	\frac{ q^{-\frac{1}{3}} \big( C\Psi_{4,2}(q^2)-qC\Psi_{4,0}(q^2) \big) }
	{ q^{-\frac{49}{3}} \big( C\Psi_{4,2}(q^{98})-q^{49}C\Psi_{4,0}(q^{98}) \big) }.
\end{align*}
Construct the function sequence $\{ K_{\alpha} \}_{\alpha=0}^\infty $ similar to $\{ L_{\alpha} \}_{\alpha=0}^\infty $:
\begin{align*}
	K_0:=1 , \quad
	K_{2\alpha+1}:= U_7
	\big( A'  K_{2\alpha}  \big)
	, \quad
	K_{2\alpha+2}:= U_7
	\big(  K_{2\alpha+1} \big) ,  \quad \quad \alpha \in \mathbb{N}.
\end{align*}
\begin{lemma} For $\alpha \ge 1 $, we have
	\begin{align}
		K_{2\alpha-1}
		&=\frac{q^3}{ \big( C\Psi_{4,2}(q^{14}) - q^7 C\Psi_{4,0}(q^{14}) \big)}
		\label{K2alpha-1} \quad \\ &\quad  \cdot \bigg[
		\sum_{n=0}^\infty
		c \psi_{4,2} \bigg(7^{2\alpha-1}n + \lambda^{(4,2)}_{2\alpha-1} \bigg)
		q^{2n+1}
		-\sum_{n=0}^\infty c \psi_{4,0} \bigg(7^{2\alpha-1}n + \lambda^{(4,0)}_{2\alpha-1} \bigg)
		q^{2n} \bigg]; \quad \quad \quad \quad  \nonumber \\
		K_{2\alpha}
		&= \frac{q}{ \big( C\Psi_{4,2}(q^2) - q C\Psi_{4,0}(q^2) \big)}
		\quad \label{K2alpha} \\ &\quad  \cdot \bigg[
		\sum_{n=0}^\infty
		c \psi_{4,2} \bigg(7^{2\alpha}n + \lambda^{(4,2)}_{2\alpha} \bigg)
		q^{2n+1}
		-\sum_{n=0}^\infty c \psi_{4,0} \bigg(7^{2\alpha}n + \lambda^{(4,0)}_{2\alpha} \bigg)
		q^{2n} \bigg], \quad \quad \quad \quad \nonumber
	\end{align}
	where 
	\begin{align}
		\label{r:lam0}
		\lambda_{\alpha}^{(4,2)}
		= \frac{1+5 \cdot 7^{\alpha}}{8}
		\ 	\ , \ \
		\lambda_{\alpha}^{(4,0)}
		= \frac{-1+ 7^{\alpha}}{3}.
	\end{align}
\end{lemma}
\begin{proof}
	By mathematical induction.
\end{proof}

\begin{lemma}
	$($\cite[Example 1.2]{ncpsi}$)$ \label{lemma:f41f42f40}	
	We have
	\begin{align*}
		\begin{pmatrix} f_{4,0} \\ f_{4,1} \\ f_{4,2} \end{pmatrix} \Bigg|_{-\frac{1}{2}}
		\begin{pmatrix} 1 & 1 \\ 0 & 1
		\end{pmatrix} &=
		\begin{pmatrix} -\frac{1}{2} + \frac{\sqrt{3}}{2} \mathrm{i} & 0 & 0 \\ 0 & \frac{\sqrt{6} - \sqrt{2}}{4} + \frac{\sqrt{6} + \sqrt{2}}{4} \mathrm{i} & 0 \\ 0 & 0 & \frac{1}{2} - \frac{\sqrt{3}}{2} \mathrm{i}
		\end{pmatrix} \cdot
		\begin{pmatrix} f_{4,0} \\ f_{4,1} \\ f_{4,2} \end{pmatrix}, \\
		\begin{pmatrix} f_{4,0} \\ f_{4,1} \\ f_{4,2} \end{pmatrix} \Bigg|_{-\frac{1}{2}}
		\begin{pmatrix} 0 & -1 \\ 1 & 0
		\end{pmatrix} &=
		\begin{pmatrix} \frac{\sqrt{2}}{4} + \frac{\sqrt{2}}{4} \mathrm{i} & -\frac{\sqrt{2}}{2} - \frac{\sqrt{2}}{2} \mathrm{i} & \frac{\sqrt{2}}{4} + \frac{\sqrt{2}}{4} \mathrm{i} \\ -\frac{\sqrt{2}}{4} - \frac{\sqrt{2}}{4} \mathrm{i} & 0 & \frac{\sqrt{2}}{4} + \frac{\sqrt{2}}{4} \mathrm{i} \\ \frac{\sqrt{2}}{4} + \frac{\sqrt{2}}{4} \mathrm{i} & \frac{\sqrt{2}}{2} + \frac{\sqrt{2}}{2} \mathrm{i} & \frac{\sqrt{2}}{4} + \frac{\sqrt{2}}{4} \mathrm{i}
		\end{pmatrix} \cdot
		\begin{pmatrix} f_{4,0} \\ f_{4,1} \\ f_{4,2} \end{pmatrix},
	\end{align*}
	where $f_{k,\beta}=q^{k/12-\beta^2/2k}C\Psi_{k,\beta}$. Especially, we have
	\begin{align*}
		f_{4,1} \bigg\vert_{-\frac{1}{2}}\tbtMat{0}{-1}{1}{0} = \frac{\sqrt{2} + \sqrt{2} \ \mathrm{i} }{4} \bigg( f_{4,2} - f_{4,0} \bigg).
	\end{align*}
\end{lemma}
The notation $f\vert_{k}\tbtmat{a}{b}{c}{d}$ means the weight $k$ modular transformation, that is, $f\vert_{k}\tbtmat{a}{b}{c}{d}(\tau):=(c\tau+d)^{k}f\left(\frac{a\tau+b}{c\tau+d}\right)$ where $(c\tau+d)^{1/2}$ is the principal branch (the argument belongs to $(-\tfrac{\uppi}{2},\tfrac{\uppi}{2}]$).

\begin{lemma} \label{lemma012} We have
	\begin{align*}
		q^{-\frac{1}{3}} \big( C\Psi_{4,2}(q^2) - q C\Psi_{4,0}(q^2) \big)= \frac{\eta_1^6}{\eta_2^7} .
	\end{align*}
\end{lemma}
\begin{proof}
	We can calculate that
	\begin{align*}
		f_{4,1} \bigg\vert_{-\frac{1}{2}}\tbtMat{0}{-1}{1}{0}
		=
		4 \cdot \frac{\eta_2^6}{\eta_1^7} \bigg\vert_{-\frac{1}{2}}\tbtMat{0}{-1}{1}{0}
		=
		\frac{\sqrt{2} + \sqrt{2} \ \mathrm{i} }{4} \cdot  \frac{\eta(\tau/2)^6}{\eta_1^7}.
	\end{align*}
	Recall Lemma \ref{lemma:f41f42f40},
	\begin{align*}
		f_{4,1} \bigg|_{-\frac{1}{2}} \begin{pmatrix} 0 & -1 \\ 1 & 0 \end{pmatrix} = \frac{\sqrt{2} + \sqrt{2} \, \mathrm{i}}{4} \Big( f_{4,2} - f_{4,0} \Big),
	\end{align*}
	and substitute $C\Psi_{4,0}(q^2) = q^{-\frac{2}{3}}f_{4,0}(2\tau)$, $C\Psi_{4,2}(q^2) = q^{\frac{1}{3}}f_{4,2}(2\tau)$, we have
	\begin{align*}
		q^{-\frac{1}{3}} \big( C\Psi_{4,2}(q^2) - q C\Psi_{4,0}(q^2) \big) =
		\Big( f_{4,2} - f_{4,0} \Big)(2\tau)
		= \frac{\eta_1^6}{\eta_2^7}.
	\end{align*}
\end{proof}
Then we have
\begin{align*}	
	A' =
	\frac{ q^{-\frac{1}{3}} \big( C\Psi_{4,2}(q^2)-qC\Psi_{4,0}(q^2) \big) }
	{ q^{-\frac{49}{3}} \big( C\Psi_{4,2}(q^{98})-q^{49}C\Psi_{4,0}(q^{98}) \big) }
	=
	\frac{\eta_1^6 \eta_{98}^7 }{\eta_2^7 \eta_{49}^6}.
\end{align*}

\begin{deff} \label{deff:Atkin} $($\cite[Section 2]{We}$)$
	Suppose
	$$
	\Gamma_0(N) := \left\{ \begin{pmatrix} a & b \\ c & d \end{pmatrix} \in \text{SL}_2(\mathbb{Z}) \mid c \equiv 0 \pmod{N} \right\}
	$$
	where $N$ is a positive integer. If $e \mid \mid N$, we call the matrix
	\begin{equation}
		W_e = \begin{pmatrix} ae & b \\ cN & de \end{pmatrix}, \quad a, b, c, d \in \mathbb{Z}, \, \det(W_e) = e,
	\end{equation}
	an \textit{Atkin-Lehner involution} of $\Gamma_0(N)$.
\end{deff}

\begin{lemma} $($\cite[Corollary 2.2]{We}$)$
	\label{corollary2.2}
	Let $W_e$ be an Atkin-Lehner involution of $\Gamma_{0}(N)$. Let $t > 0$ be such that $t \mid N$. Then
	\begin{align*}
		\eta\Big(t W_e \tau \Big)
		=\eta\bigg(t \cdot \frac{ae\tau+b}{cN \tau +de} \bigg)
		= \nu_{\eta}\big(M\big)
		\bigg(\frac{cN \tau +de}{\delta} \bigg)^{\frac{1}{2}}
		\eta  \bigg(\frac{et}{\delta^2} \tau \bigg),
	\end{align*}
	where $\delta= (e,t)$, $\nu_{\eta}$ is eta-multiplier and
	\begin{align*}
		M=
		\begin{pmatrix}
			a \delta & b t / \delta \\
			cN \delta / et  & de / \delta
		\end{pmatrix}
		\in \slZ.
	\end{align*}
\end{lemma}

\begin{lemma} \label{lemma:A'=AW}
	Let $W=W_e=\tbtMat{25 \cdot 2 }{1}{1 \cdot 98}{1 \cdot 2}$, then $A'=A|W$ where the notation $f \big\vert W := f \big\vert_0 W$.
\end{lemma}
\begin{proof}
	Check that $ e=2 \mid N=98 $ and $\det{(W_e)}=e=2$. According to Definiton \ref{deff:Atkin}, $W$ is an Atkin-Lehner involution of $\Gamma_{0}(98)$. Then  we can calculate that
	\begin{align*}
		A \big\vert W
		= \frac{\eta_2^6 \eta_{49}^7 }{\eta_1^7 \eta_{98}^6}
		\bigg\vert W
		= \frac{\eta_1^6 \eta_{98}^7 }{\eta_2^7 \eta_{49}^6}
		=  A'
	\end{align*}
	by Lemma \ref{corollary2.2}.
\end{proof}

\begin{lemma} $($\cite[Lemma 6]{UWandWU}$)$ \label{lemma:UWandWU}
	Let $p$ be prime, $p \mid N$, $e \parallel N$ and $(p,e)=1$. If $f(\tau)$ is a modular function on $\Gamma_{0}(N)$ then
	\begin{align*}
		U_p( f ) \big\vert W_e = U_p\big( f \vert W_e \big).
	\end{align*}
\end{lemma}

\begin{rema} \label{rema98}
	We can check that $A$ is a modular function on $\Gamma_{0}(98)$ by \cite[Theorem 24]{nm59}. Let $p$ be prime. It is a basic fact that if $f$ is a modular function on $\Gamma_0(pN)$ and $p \mid N$, then $U_p(f)$ is a modular function on $\Gamma_0(N)$. Hence, by induction of $\alpha$, $L_\alpha$ is a modular function on $\Gamma_0(14)$ for all $\alpha \geq 0$.
\end{rema}

\begin{lemma}  \label{lemma:K=LW}
	For $\alpha \in \mathbb{N}$
	\begin{align*}	
		K_\alpha = L_\alpha \big\vert W.
	\end{align*}
\end{lemma}

\begin{proof}
	$K_0=L_0 \big\vert W$. Suppose for a fixed $\beta \in \mathbb{N}$, we have $K_{2\beta} = L_{2\beta} \big\vert W$, then
	\begin{align*}
		L_{2\beta+1} \big\vert W = U_7 \big( A L_{2\beta} \big) \big\vert W = U_7 \big( A L_{2\beta} \big\vert W \big)
		= U_7 \big( A' K_{2\beta}  \big) = K_{2\beta+1}.
	\end{align*}
	The second equation is based on Lemma \ref{lemma:UWandWU} and Remark \ref{rema98}. The third equation is based on Lemma \ref{lemma:A'=AW}. Suppose for a fixed $\beta \in \mathbb{N}$, we have $K_{2\beta+1} = L_{2\beta+1} \big\vert W$, then
	\begin{align*}
		L_{2\beta+2} \big\vert W = U_7 \big(  L_{2\beta+1} \big) \big\vert W = U_7 \big( L_{2\beta+1} \big\vert W \big)
		= U_7 \big( K_{2\beta+1}  \big) = K_{2\beta+2}.
	\end{align*}
	Hence for all $\alpha \in \mathbb{N}$, $K_\alpha = L_\alpha \big\vert W$.
\end{proof}

Now define
\begin{align*}
	\overline{t}:=t \big\vert W
	= \frac{\eta_{14}^4}{\eta_2^4},
	\ \ \ \
	\overline{p_1}:=p_1\big\vert W
	=\frac{\eta_{7}^4\eta_2^4}{\eta_{14}^4\eta_1^4},
	\ \ \ \
	\overline{p_2}:=p_2\big\vert W
	=8\frac{\eta_{14}^3\eta_2^3}{\eta_{7}^3\eta_1^3}-\frac{\eta_{7}^4\eta_2^4}{\eta_{14}^4\eta_1^4},
\end{align*}
and  {\small
	\begin{align*}
		\overline{X^{(0)}} :=\big\{ f \big\vert W : f \in X^{(0)} \big\}&= \bigg\{
		\ \  \ \ \  r_1(0) +7r_1(1) \ \overline{t} \ \ +
		\sum_{n=2}^\infty r_1(n)  7^{ \lfloor \frac{7n-2}{4} \rfloor } \overline{t}^n \\
		&  \quad \quad \quad \quad \quad \quad
		+7r_2(1)\overline{p_1}\overline{t}+
		\sum_{n=2}^\infty r_2(n)  7^{ \lfloor \frac{7n-2}{4} \rfloor } \overline{p_1} \overline{t}^n \\
		&  \quad \quad \quad \quad \quad \quad
		+r_3(1)\overline{p_2}\overline{t}+
		\sum_{n=2}^\infty r_3(n)  7^{ \lfloor \frac{7n-6}{4} \rfloor } \overline{p_2} \overline{t}^n
		\bigg\}, \\
		\overline{X^{(1)}} :=\big\{ f \big\vert W : f \in X^{(1)} \big\}&=\bigg\{
		\quad \quad \quad \quad \quad
		+\ r_1(1) \ \overline{t} \  +
		\sum_{n=2}^\infty r_1(n)  7^{ \lfloor \frac{7n-7}{4} \rfloor } \overline{t}^n \\
		&  \ \ \ \ \ \ \ \ \ +r_2(0)\overline{p_1}+r_2(1)\overline{p_1}\overline{t}+
		\sum_{n=2}^\infty r_2(n)  7^{ \lfloor \frac{7n-7}{4} \rfloor } \overline{p_1} \overline{t}^n \\
		&  \ \ \ \ \ \ \ \ \ +r_3(0)\overline{p_2}+r_3(1)\overline{p_2}\overline{t}+
		\sum_{n=2}^\infty r_3(n)  7^{ \lfloor \frac{7n-11}{4} \rfloor } \overline{p_2} \overline{t}^n
		\bigg\},
\end{align*} }
where $r_1,r_2,r_3:\mathbb{Z} \to \mathbb{Z}$ are arbitrarily selected discrete functions.

\begin{thm} \label{thm:K=7g}
	For each $\alpha \ge 1$, there exist $g_{2 \alpha-1 }\in \overline{X^{(1)}}$ and $g_{2 \alpha }\in \overline{X^{(0)}}$ such that
	\begin{align*}
		&K_{2 \alpha-1}= 7^{\alpha} g_{2\alpha-1}, \\
		&K_{2 \alpha}= 7^{\alpha} g_{2\alpha}. \
	\end{align*}
\end{thm}

\begin{proof}
	By Corollary \ref{r:col}, for each $\alpha \ge 1$, there exist $f_{2 \alpha-1 }\in X^{(1)}$ and $f_{2 \alpha }\in X^{(0)}$ such that
	\begin{align*}
		L_{2 \alpha-1}= 7^{\alpha} f_{2\alpha-1}, \quad \quad
		L_{2 \alpha}= 7^{\alpha} f_{2\alpha}. \
	\end{align*}
	Then by Lemma \ref{lemma:K=LW},
	\begin{align*}
		K_{2 \alpha-1} = 7^{\alpha} f_{2\alpha-1} \big\vert W = 7^{\alpha} g_{2\alpha-1},
		\quad \quad
		K_{2 \alpha} = 7^{\alpha} f_{2\alpha} \big\vert W = 7^{\alpha} g_{2\alpha}. \
	\end{align*}
\end{proof}

Combining \eqref{K2alpha-1}, \eqref{K2alpha} and Theorem \ref{thm:K=7g}, we conclude the proof of Theorems \ref{thm:cpsi40} and \ref{thm:cpsi42}.

\begin{thm} \label{thm:cpsi40}
	For $\alpha \in \mathbb{N}^*$ and $n \in \mathbb{N}$, we have
	\begin{align*}
		c\psi_{4,0}(7^\alpha n+\lambda_\alpha^{(4,0)}) \equiv 0 \pmod{7^{\lfloor\frac{\alpha+1}{2}\rfloor}},
	\end{align*}
	where $\lambda_\alpha^{(4,0)}$ defined by \eqref{r:lam0}.
\end{thm}

%

\begin{thm} \label{thm:cpsi42}
	For $\alpha \in \mathbb{N}^*$ and $n \in \mathbb{N}$, we have
	\begin{align*}
		c\psi_{4,2}(7^\alpha n+\lambda_\alpha^{(4,2)}) \equiv 0 \pmod{7^{\lfloor\frac{\alpha+1}{2}\rfloor}},
	\end{align*}
	where $\lambda_\alpha^{(4,2)}$ defined by \eqref{r:lam0}.
\end{thm}
%
Theorems \ref{thm:cpsi41}, \ref{thm:cpsi40}, and \ref{thm:cpsi42} together provide Theorem \ref{thm:cpsi4012}.

\appendix

\section*{Appendix}

{\tiny 	
	Group I:
	\begin{align*}
		U_7(At^0)
		&=\quad \quad 16\cdot7^2\cdot t+7148\cdot7^2\cdot t^2+1536\cdot 7^5\cdot t^3+921\cdot7^7\cdot t^4+40\cdot 7^{10}\cdot t^5+324\cdot 7^{10}\cdot t^6+4\cdot 7^{13}\cdot t^7+7^{14}\cdot t^8\\
		&+{p_1} \big( \quad 2\cdot 7^2+
		6161\cdot7\cdot t+9045\cdot 7^3\cdot t^2+36870\cdot 7^4\cdot t^3+1551\cdot 7^7\cdot t^4+248\cdot 7^9\cdot t^5+145\cdot 7^{10}\cdot t^6+5\cdot 7^{12}\cdot t^7
		\quad \big)  \\
		&+{p_2} \big(\quad
		2\cdot 7^2+6273\cdot 7\cdot t+70359\cdot 7^2\cdot t^2+330714\cdot 7^3\cdot t^3+16935\cdot 7^6\cdot t^4+3540\cdot 7^8\cdot t^5+3053\cdot 7^9\cdot t^6
		\\
		&\quad \quad \quad
		+207\cdot 7^{11}\cdot t^7+6\cdot 7^{13}\cdot t^8
		\quad  \big)
		\\			
		U_7(At^{-1})		
		&=\quad \quad \ 7^2\cdot t \\
		&+{p_1} \big( \quad		
		6 \quad	\big)  \\
		&+{p_2} \big( \quad 6+48\cdot t	 \quad	\big)
		\\
		U_7(At^{-2})		
		&=\quad \quad -8\cdot 7\cdot t+7^3\cdot t^2\\
		&+{p_1} \big( \quad	-8	-2\cdot 7\cdot t
		\quad	\big)\\
		&+{p_2} \big(\quad -8-78\cdot t \quad	\big)
		\\
		U_7(At^{-3})		
		&=\quad \quad
		\ 150\cdot 7\cdot t+24\cdot 7^3\cdot t^2+7^5\cdot t^3
		\\
		&+{p_1} \big(\quad
		20\cdot 7+30\cdot 7^2\cdot t+10\cdot 7^3\cdot t^2
		\quad	\big)
		\\
		&+{p_2} \big(\quad 20\cdot 7
		+370\cdot 7\cdot t
		+302\cdot 7^2\cdot t^2+12\cdot 7^4\cdot t^3  \quad	\big)
		\\		
		U_7(At^{-4})
		&=\quad \quad \ 8-24\cdot 7^3\cdot t-160\cdot 7^3\cdot t^2-7^7\cdot t^4
		\\
		&+{p_1} \big(\quad	t^{-1}-156\cdot 7-191\cdot7^2\cdot t+7^5\cdot t^3
		\quad	\big)
		\\
		&+{p_2} \big(\quad  t^{-1}-1084
		-2585\cdot7\cdot t-216\cdot 7^3\cdot t^2
		-7^4\cdot t^3  \quad	\big)
		\\			
		U_7(At^{-5})	
		&=\quad \quad \      -104+144\cdot 7^3\cdot t+32\cdot 7^4\cdot t^2-16\cdot 7^6\cdot t^3+20\cdot 7^7\cdot t^4+7^9\cdot t^5\\
		&+{p_1} \big(\quad -13\cdot t^{-1}+6574+423\cdot 7^2\cdot t	-134\cdot 7^4\cdot t^2+3\cdot 7^6\cdot t^3+6\cdot 7^7\cdot t^4	\quad	\big)\\
		&+{p_2} \big(\quad
		-13\cdot t^{-1}+6470+10475\cdot 7\cdot t-446\cdot 7^3\cdot t^2-129\cdot 7^5\cdot t^3+10\cdot 7^7\cdot t^4+6\cdot 7^8\cdot t^5
		\quad	\big)
		\\			
		U_7(At^{-6})
		&=\quad \quad \  +904-4544\cdot 7^2\cdot t+1066\cdot 7^4\cdot t^2+192\cdot 7^6\cdot t^3-400\cdot 7^7\cdot t^4-40\cdot 7^9\cdot t^5-7^{11}\cdot t^6     \\
		&+{p_1} \big(\quad	113\cdot t^{-1}-4294\cdot 7+5685\cdot 7^2\cdot t+1382\cdot 7^4\cdot t^2-295\cdot 7^6\cdot t^3-46\cdot 7^8\cdot t^4-10\cdot 7^9\cdot t^5	\quad	\big)\\
		&+{p_2} \big(\quad  113\cdot t^{-1}-29154+5427\cdot 7\cdot t+2290\cdot 7^4\cdot t^2-433\cdot 7^5\cdot t^3-654\cdot 7^7\cdot t^4-430\cdot 7^8\cdot t^5-12\cdot 7^{10}\cdot t^6
		\quad	\big)
	\end{align*}
	
	Group II:
	\begin{align*}		
		U_7(A{p_1}t^{0})
		&=\quad \quad \  8 \cdot 7\cdot t+12\cdot 7^3\cdot t^2+24\cdot 7^4\cdot t^3+2\cdot 7^6\cdot t^4   \\
		&+{p_1} \big(\quad	7+79\cdot 7\cdot t+190\cdot 7^2\cdot t^2+24\cdot 7^4\cdot t^3+7^6\cdot t^4	\quad	\big)\\
		&+{p_2} \big(\quad  7+87\cdot 7\cdot t+40\cdot 7^3\cdot t^2+381\cdot 7^3\cdot t^3+33\cdot 7^5\cdot t^4+7^7\cdot t^5   \quad	\big)
		\\			
		U_7(A{p_1}t^{-1})
		&=\quad \quad \  2\cdot 7\cdot t   \\
		&+{p_1} \big(\quad	2+7\cdot t	\quad	\big)\\
		&+{p_2} \big(\quad  2 +23\cdot t+7^2\cdot t^2	\quad	\big)
		\\			
		U_7(A{p_1}t^{-2})
		&=\quad \quad \  -4\cdot 7^2\cdot t-4\cdot 7^3\cdot t^2   \\
		&+{p_1} \big(\quad -26-32\cdot 7\cdot t-7^3\cdot t^2		\quad	\big)\\
		&+{p_2} \big(\quad -26-432\cdot t
		-43\cdot 7^2\cdot t^2-7^4\cdot t^3		\quad	\big)
		\\			
		U_7(A{p_1}t^{-3})
		&=\quad \quad \ 26\cdot 7^2\cdot t+8\cdot 7^3\cdot t^2-4\cdot 7^5\cdot t^3    \\
		&+{p_1} \big( \quad 24\cdot 7+94\cdot 7\cdot t-20\cdot 7^3\cdot t^2	-7^5\cdot t^3	\quad	\big)
		\\
		&+{p_2} \big(\quad 24\cdot 7
		+286\cdot 7\cdot t
		-34\cdot 7^2\cdot t^2
		-31\cdot 7^4\cdot t^3-7^6\cdot t^4		\quad	\big)
		\\			
		U_7(A{p_1}t^{-4})
		&=\quad \quad \   -124\cdot 7^2\cdot t+32\cdot 7^4\cdot t^2+80\cdot 7^5\cdot t^3+2\cdot 7^7\cdot t^4  \\
		&+{p_1} \big( \quad -804+159\cdot 7^2\cdot t+470\cdot 7^3\cdot t^2+24\cdot 7^5\cdot t^3		\quad	\big)
		\\
		&+{p_2} \big(\quad -804
		+194\cdot 7\cdot t
		+652\cdot 7^3\cdot t^2
		+717\cdot 7^4\cdot t^3+26\cdot 7^6\cdot t^4		\quad	\big)
		\\			
		U_7(A{p_1}t^{-5})
		&=\quad \quad \ 8+300\cdot 7^2\cdot t-552\cdot 7^4\cdot t^2-20\cdot 7^7\cdot t^3-36\cdot 7^7\cdot t^4    \\
		&+{p_1} \big( \quad t^{-1}+2012-3153\cdot 7^2\cdot t	-802\cdot 7^4\cdot t^2-29\cdot 7^6\cdot t^3+18\cdot 7^7\cdot t^4+7^9\cdot t^5	\quad	\big)
		\\
		&+{p_2} \big(\quad t^{-1}+2020
		-19769\cdot 7\cdot t
		-9220\cdot 7^3\cdot t^2
		-163\cdot 7^6\cdot t^3
		-16\cdot 7^7\cdot t^4
		+25\cdot 7^8\cdot t^5+7^{10} \cdot t^6		\quad	\big)
		\\	
		U_7(A{p_1}t^{-6})
		&=\quad \quad \  -24\cdot 7+2300\cdot 7^2 \cdot t+6368\cdot 7^4\cdot t^2+1648\cdot 7^6\cdot t^3+146\cdot 7^8\cdot t^4+80\cdot 7^9\cdot t^5+4\cdot 7^{11}\cdot t^6   \\
		&+{p_1} \big( \quad -3\cdot 7\cdot t^{-1}+13854+37400\cdot 7^2\cdot t+9144\cdot 7^4\cdot t^2+675\cdot 7^6\cdot t^3+209\cdot 7^7\cdot t^4+20\cdot 7^9\cdot t^5+7^{11}\cdot t^6		\quad	\big)
		\\
		&+{p_2} \big(\quad -3\cdot 7\cdot t^{-1}+13686+277603\cdot 7\cdot t+106814\cdot 7^3\cdot t^2+15420\cdot 7^5\cdot t^3+1019\cdot 7^7\cdot t^4+425\cdot 7^8\cdot t^5
		\\
		&\quad \quad \quad
		+31\cdot 7^{10}\cdot t^6+7^{12}\cdot t^7	\quad	\big)
	\end{align*}
	
	Group III:
	\begin{align*}
		U_7(A{p_2}t^{0})
		&=\quad \quad \   64\cdot 7\cdot t+12\cdot 7^4\cdot t^2+24\cdot 7^5\cdot t^3+2\cdot 7^7\cdot t^4  \\
		&+{p_1} \big( \quad 8\cdot 7+79\cdot 7^2\cdot t+190\cdot 7^3\cdot t^2+24\cdot 7^5\cdot t^3	+7^7\cdot t^4	\quad	\big)
		\\
		&+{p_2} \big(\quad 8\cdot7
		+617\cdot 7\cdot t
		+40\cdot 7^4\cdot t^2
		+381\cdot 7^4\cdot t^3
		+33\cdot 7^6\cdot t^4+7^8\cdot t^5		\quad	\big)
		\\	
		U_7(A{p_2}t^{-1})
		&=\quad \quad \  -2\cdot 7\cdot t   \\
		&+{p_1} \big( \quad -2-2\cdot 7\cdot t		\quad	\big)
		\\
		&+{p_2} \big(\quad -2-30\cdot t
		-7^2\cdot t^2		\quad	\big)
		\\	
		U_7(A{p_2}t^{-2})
		&=\quad \quad \   36\cdot 7\cdot t+4\cdot 7^3\cdot t^2  \\
		&+{p_1} \big( \quad 36+48\cdot 7\cdot t		\quad	\big)
		\\
		&+{p_2} \big(\quad 36+624\cdot t
		+52\cdot 7^2\cdot t^2+7^4\cdot t^3		\quad	\big)
		\\	
		U_7(A{p_2}t^{-3})
		&=\quad \quad \  8-58 \cdot 7^2\cdot t-48\cdot 7^3\cdot t^2+4\cdot 7^5\cdot t^3   \\
		&+{p_1} \big( \quad t^{-1}-8\cdot 7^2-65\cdot 7^2\cdot t	+20\cdot 7^3\cdot t^2+7^5\cdot t^3	\quad	\big)
		\\
		&+{p_2} \big( \quad t^{-1}-384
		-129\cdot 7^2\cdot t
		-382\cdot 7^2\cdot t^2
		+23\cdot 7^4\cdot t^3+7^6\cdot t^4
		\quad	\big)
		\\
		U_7(A{p_2}t^{-4})
		&=\quad \quad \ -96+412\cdot 7^2\cdot t-16\cdot 7^6\cdot t^3-2\cdot 7^7\cdot t^4    \\
		&+{p_1} \big( \quad -12\cdot t^{-1}+2778+109\cdot 7^2\cdot t-100\cdot 7^4\cdot t^2	-4\cdot 7^6\cdot t^3+7^7\cdot t^4	\quad	\big)
		\\
		&+{p_2} \big(\quad -12\cdot t^{-1}+2682
		+3938\cdot 7\cdot t
		-562\cdot 7^3\cdot t^2
		-143\cdot 7^5\cdot t^3
		-5\cdot 7^7\cdot t^4 		\quad	\big)
		\\	
		U_7(A{p_2}t^{-5})
		&=\quad \quad \  800-2076\cdot 7^2\cdot t+752\cdot 7^4\cdot t^2+268\cdot 7^6\cdot t^3+68\cdot 7^7\cdot t^4   \\
		&+{p_1} \big( \quad 100\cdot t^{-1}-14268+596 \cdot 7^3\cdot t+1602\cdot 7^4\cdot t^2	+58\cdot 7^6\cdot t^3-26\cdot 7^7\cdot t^4	\quad	\big)
		\\
		&+{p_2} \big(\quad 100\cdot t^{-1}-1924\cdot 7
		+1842\cdot 7^2\cdot t
		+15860\cdot 7^3\cdot t^2
		+2304\cdot 7^5\cdot t^3
		+8\cdot 7^8\cdot t^4-26\cdot 7^8\cdot t^5
		-7^{10}\cdot t^6 		\quad	\big)
		\\	
		U_7(A{p_2}t^{-6})
		&=\quad \quad \  -776\cdot 7+4444\cdot 7^2\cdot t-12960\cdot 7^4\cdot t^2-3744\cdot 7^6\cdot t^3-274\cdot 7^8\cdot t^4-96\cdot 7^9\cdot t^5-4\cdot 7^{11}\cdot t^6   \\
		&+{p_1} \big( \quad -97\cdot 7\cdot t^{-1}+722\cdot 7^2-77076\cdot 7^2\cdot t-20798\cdot 7^4\cdot t^2-1175\cdot 7^6\cdot t^3-3\cdot 7^9\cdot t^4-44\cdot 7^9\cdot t^5-2\cdot 7^{11}\cdot t^6
		\quad	\big)
		\\
		&+{p_2} \big(\quad -97\cdot 7\cdot t^{-1}+4278\cdot 7-499047\cdot 7\cdot t-33272\cdot 7^4\cdot t^2-32808\cdot 7^5\cdot t^3-219\cdot 7^8\cdot t^4-241\cdot 7^8\cdot t^5
		\\
		&\quad \quad \quad
		-22\cdot 7^{10}\cdot t^6-7^{12}\cdot t^7	
		\quad	\big)
	\end{align*}
	Group IV
	\begin{align*}
		U_7(1) &= 1 & & \\
		U_7(t^{-1}) &= -4 - 7\cdot t & & \\
		U_7(t^{-2}) &= 20 - 7^3 \cdot t^2 & &\\
		U_7(t^{-3}) &= -88 - 7^5 \cdot t^3 & & \\
		U_7(t^{-4}) &= 260 - 7^7 \cdot t^4 & & \\
		U_7(t^{-5}) &= 68 \cdot 7 - 7^9\cdot t^5 & & \\
		U_7(t^{-6}) &= -2392 \cdot 7 \cdot - 7^{11}\cdot t^6 & &
	\end{align*}
	Group V
	\begin{align*}
		U_7(p_1) &= 4 + 1672 \cdot 7 \cdot t + 2320 \cdot 7^3 \cdot t^2 +920\cdot7^5 \cdot t^3+ 144 \cdot 7^7 \cdot t^4 + 8 \cdot 7^9 \cdot t^5 \\
		&\quad + p_1 (83 \cdot 7 \cdot t + 454 \cdot 7^3 \cdot t^2 + 407 \cdot 7^5 \cdot t^3 + 134 \cdot 7^7 \cdot t^4 + 19 \cdot 7^9 \cdot t^5 + 7^{11} \cdot t^6) \\
		&\quad + p_2 (-4\cdot t-34 \cdot 7^3 \cdot t^2 - 327 \cdot 7^4 \cdot t^3 - 18 \cdot 7^7 \cdot t^4 - 19 \cdot 7^8 \cdot t^5 - 7^{10} \cdot t^6) \\[1ex]
		U_7(p_1 t^{-1}) &= 8 \\
		&\quad + p_1 (-7 \cdot t) \\
		&\quad + p_2 (-7 \cdot t - 7^2 \cdot t^2) \\[1ex]
		U_7(p_1 t^{-2}) &= -12 \\
		&\quad + p_1 (-7^3 \cdot t^2) \\
		&\quad + p_2 (-7^3 \cdot t^2 - 7^4 \cdot t^3) \\[1ex]
		U_7(p_1 t^{-3}) &= 80 - 64 \cdot 7^2 \cdot t - 32 \cdot 7^4 \cdot t^2 - 4 \cdot 7^6 \cdot t^3 \\
		&\quad + p_1 (1 + 6 \cdot 7^2 \cdot t + 6 \cdot 7^4 \cdot t^2 + 6 \cdot 7^5 \cdot t^3) \\
		&\quad + p_2 (4 \cdot 7 \cdot t + 3 \cdot 7^4 \cdot t^2 + 8 \cdot 7^5 \cdot t^3 + 6 \cdot 7^6 \cdot t^4) \\[1ex]
		U_7(p_1 t^{-4}) &= -96 \cdot 7 + 384 \cdot 7^2 \cdot t + 192 \cdot 7^4 \cdot t^2 + 24 \cdot 7^6 \cdot t^3 \\
		&\quad + p_1 (-6 - 36 \cdot 7^2 \cdot t - 36 \cdot 7^4 \cdot t^2 - 6 \cdot 7^6 \cdot t^3 - 7^7 \cdot t^4) \\
		&\quad + p_2 (-24 \cdot 7 \cdot t - 18 \cdot 7^4 \cdot t^2 - 54 \cdot 7^5 \cdot t^3 - 7^8 \cdot t^4 - 7^8 \cdot t^5) \\[1ex]
		U_7(p_1 t^{-5}) &= 752 \cdot 7 - 192 \cdot 7^3 \cdot t - 96 \cdot 7^5 \cdot t^2 - 12 \cdot 7^7 \cdot t^3 \\
		&\quad + p_1 (3 \cdot 7 + 18 \cdot 7^3 \cdot t + 18 \cdot 7^5 \cdot t^2 + 3 \cdot 7^7 \cdot t^3 - 7^9 \cdot t^5) \\
		&\quad + p_2 (12 \cdot 7^2 \cdot t + 9 \cdot 7^5 \cdot t^2 + 27 \cdot 7^6 \cdot t^3 + 3 \cdot 7^8 \cdot t^4 - 7^9 \cdot t^5 - 7^{10} \cdot t^6) \\[1ex]
		U_7(p_1 t^{-6}) &= -108 \cdot 7^3 - 128 \cdot 7^3 \cdot t - 64 \cdot 7^5 \cdot t^2 - 8 \cdot 7^7 \cdot t^3 \\
		&\quad + p_1 (2 \cdot 7 + 12 \cdot 7^3 \cdot t + 12 \cdot 7^5 \cdot t^2 + 2 \cdot 7^7 \cdot t^3 - 7^{11} \cdot t^6) \\
		&\quad + p_2 (8 \cdot 7^2 \cdot t + 6 \cdot 7^5 \cdot t^2 + 18 \cdot 7^6 \cdot t^3 + 2 \cdot 7^8 \cdot t^4 - 7^{11} \cdot t^6 - 7^{12} \cdot t^7)
	\end{align*}
	Group VI
	\begin{align*}
		U_7(p_2) &= -4 \\
		&\quad + p_2 \cdot t \\[1ex]
		U_7(p_2 t^{-1}) &= 16 \\
		&\quad + p_2 (7^2 \cdot t^2) \\[1ex]
		U_7(p_2 t^{-2}) &= -76 \\
		&\quad + p_2 (7^4 \cdot t^3) \\[1ex]
		U_7(p_2 t^{-3}) &= 304 + 64 \cdot 7^2 \cdot t + 32 \cdot 7^4 \cdot t^2 + 4 \cdot 7^6 \cdot t^3 \\
		&\quad + p_1 (-1 - 6 \cdot 7^2 \cdot t - 6 \cdot 7^4 \cdot t^2 - 7^6 \cdot t^3) \\
		&\quad + p_2 (-4 \cdot 7 \cdot t - 3 \cdot 7^4 \cdot t^2 - 9 \cdot 7^5 \cdot t^3 - 6 \cdot 7^6 \cdot t^4) \\[1ex]
		U_7(p_2 t^{-4}) &= -88 \cdot 7 - 128 \cdot 7^3 \cdot t - 64 \cdot 7^5 \cdot t^2 - 8 \cdot 7^7 \cdot t^3 \\
		&\quad + p_1 (2 \cdot 7 + 12 \cdot 7^3 \cdot t + 12 \cdot 7^5 \cdot t^2 + 2 \cdot 7^7 \cdot t^3) \\
		&\quad + p_2 (8 \cdot 7^2 \cdot t + 6 \cdot 7^5 \cdot t^2 + 18 \cdot 7^6 \cdot t^3 + 2 \cdot 7^8 \cdot t^4 + 7^8 \cdot t^5) \\[1ex]
		U_7(p_2 t^{-5}) &= -104 \cdot 7^2 + 1216 \cdot 7^3 \cdot t + 608 \cdot 7^5 \cdot t^2 + 76 \cdot 7^7 \cdot t^3 \\
		&\quad + p_1 (-19 \cdot 7 - 114 \cdot 7^3 \cdot t - 114 \cdot 7^5 \cdot t^2 - 19 \cdot 7^7 \cdot t^3) \\
		&\quad + p_2 (-76 \cdot 7^2 \cdot t - 57 \cdot 7^5 \cdot t^2 - 171 \cdot 7^6 \cdot t^3 - 19 \cdot 7^8 \cdot t^4 + 7^{10} \cdot t^6) \\[1ex]
		U_7(p_2 t^{-6}) &= 12588 \cdot 7 - 9288 \cdot 7^3 \cdot t - 4432 \cdot 7^5 \cdot t^2 - 464 \cdot 7^7 \cdot t^3 + 144 \cdot 7^8 \cdot t^4 + 8 \cdot 7^{10} \cdot t^5 \\
		&\quad + p_1 (149 \cdot 7 + 905 \cdot 7^3 \cdot t + 958 \cdot 7^5 \cdot t^2 + 207 \cdot 7^7 \cdot t^3 + 134 \cdot 7^8 \cdot t^4 + 19 \cdot 7^{10} \cdot t^5 + 7^{12} \cdot t^6) \\
		&\quad + p_2 (1 + 85 \cdot 7^3 \cdot t + 3092 \cdot 7^4 \cdot t^2 + 1293 \cdot 7^6 \cdot t^3 + 916 \cdot 7^7 \cdot t^4 - 19 \cdot 7^9 \cdot t^5 - 7^{11} \cdot t^6 + 7^{12} \cdot t^7)
	\end{align*}
}

\bibliographystyle{plain}
\bibliography{cpsi4012}

\end{document}